\documentclass[12pt]{amsart}
\usepackage{amssymb,amsmath,tikz-cd,amsfonts,amsthm}
\usepackage[colorlinks]{hyperref}
\usepackage[all]{xy}

\newcommand{\cC}{\mathcal C}
\newcommand{\cF}{\mathcal F}

\newcommand{\cH}{\mathcal H}

\newcommand{\cU}{\mathcal U}
\newcommand{\oU}{\overline\cU}

\newcommand{\bC}{\mathbb C}

\newcommand{\bP}{\mathbb P}

\newcommand{\bZ}{\mathbb Z}

\newcommand{\cov}{\widetilde}

\DeclareMathOperator{\Bl}{\operatorname{Bl}}

\DeclareMathOperator{\CH}{\operatorname{CH}}

\DeclareMathOperator{\chara}{\operatorname{char}}
\DeclareMathOperator{\Gal}{\operatorname{Gal}}
\DeclareMathOperator{\Gr}{\operatorname{Gr}}
\DeclareMathOperator{\Hilb}{Hilb}

\DeclareMathOperator{\NS}{\operatorname{NS}}
\DeclareMathOperator{\Num}{\operatorname{Num}}

\DeclareMathOperator{\Pic}{\operatorname{Pic}}

\DeclareMathOperator{\rank}{rank}

\DeclareMathOperator{\Sym}{Sym}

\newtheorem{prop}{Proposition}
\numberwithin{prop}{section}
\newtheorem{theo}[prop]{Theorem}
\newtheorem{coro}[prop]{Corollary}
\newtheorem{lemm}[prop]{Lemma}

\theoremstyle{remark}
\newtheorem{rema}[prop]{Remark}

\newtheorem{notation}[prop]{Notation}

\theoremstyle{definition}
\newtheorem{defi}[prop]{Definition}

\author[Brooke]{Corey Brooke}
\address{Department of Mathematics \\
University of Oregon \\
Eugene, OR 97403-1222 \\
USA}
\email{cbrooke@uoregon.edu}

\title[]{Lines on cubic threefolds and fourfolds containing a plane}

\begin{document}

\maketitle

\begin{abstract}
We describe the Fano scheme of lines on a general cubic threefold containing a plane over a field $k$ of characteristic different from 2. Then, we use the Fano scheme to characterize rationality for such cubic threefolds over nonclosed fields and to construct a Lagrangian fibration from the Fano variety of lines on a cubic fourfold containing a plane explicitly.
\end{abstract}

\section{Introduction}\label{sect:intro}

The Fano scheme of lines $F\subset\Gr(2,n+2)$ on a cubic $n$-fold $X\subset\bP^{n+1}$ has proved an important tool for understanding the geometry of $X$, famously illustrated by Clemens and Griffith's proof that a smooth, complex cubic threefold is irrational \cite{clemens}. On the other hand, the Fano scheme $F$ is often geometrically interesting in its own right: for example, Fano varieties of smooth cubic fourfolds provide concrete examples of hyperk\"ahler fourfolds \cite{beauville}. This paper explains the geometry of the Fano scheme $F(Y)$ of lines on a general cubic threefold $Y$ containing a plane $P$ and presents two applications, one to the birational geometry of $Y$ over a nonclosed field and another to the construction of Lagrangian fibrations from hyperk\"ahler fourfolds. 

Unlike the cubic threefolds considered by Clemens and Griffiths, a cubic threefold containing a plane is always singular, as is its Fano scheme. We give the following description of the Fano scheme:

\begin{theo}\label{theo:fanothree}
Let $Y$ be a cubic threefold containing exactly one plane $P$ over a field of characteristic not $2$, and suppose the singularities of $Y$ are isolated. Then the Fano scheme $F(Y)$ of lines on $Y$ consists of
\begin{enumerate}
    \item the plane $P^*$ dual to $P$,
    \item a ruled surface $\cF$ over a curve $C$ of arithmetic genus $2$, and
    \item a singular surface $\oU$, geometrically birational to $\Hilb^2C$.
\end{enumerate}
If $Y\setminus P$ has fewer than three nodes, then $C$ is irreducible, as is each component of $F$ described above. If $Y\setminus P$ is smooth, then so is $C$, and $\oU$ is birational to a $\Pic^0_C$-torsor $T$ satisfying the relation $2[T]=[\Pic^1_C]$ in the Weil-Ch\^atelet group.
\end{theo}

A similar account of the birational geometry of $F(Y)$ over $\bC$ appears in \cite[Sec. 3.3]{hadan}. The description of the surface $T$ and its arithmetic, however, are novel.

Over a separably closed field, a cubic threefold $Y$ containing a plane $P$ is rational. There are two classical constructions: projection from a node on $Y$ and projection onto $L\times P$ where $L\subset Y\setminus P$ is a line. Over an arbitrary field $k$, however, neither construction may be available; in fact, it turns out that $Y$ has a $k$-rational node along $P$ or line in $Y\setminus P$ exactly when the surface $T$ has a $k$-point. Using intermediate Jacobian torsor obstructions developed by Hassett and Tschinkel in \cite{hassetttschinkelrationality} and by Benoist and Wittenberg in \cite{benoistwitt}, we find that the torsor $T$ completely controls the rationality of $Y$:

\begin{theo}\label{theo:rationality}
Let $Y$ be a cubic threefold containing a plane $P$ over a field $k$ of characteristic not $2$, and suppose that $Y$ has isolated singularities, all of which are confined to $P$. Then $Y$ is rational over $k$ if and only if $Y$ contains a $k$-rational node or a $k$-rational line in $Y\setminus P$.
\end{theo}

The second application of Theorem~\ref{theo:fanothree} is an explicit construction of a rational Lagrangian fibration from the Fano variety $F$ of lines on a smooth cubic fourfold $X$ containing a plane $P$. When $k=\bC$, tools from hyperk\"ahler geometry detect that $F$ is birational to another hyperk\"ahler fourfold $M$ admitting a Lagrangian fibration $M\to\bP^2$, as explained in Section~\ref{sect:lagrangian}. We construct this fibration explicitly in a way that works over any field of characteristic different from $2$:

\begin{theo}\label{theo:lagrangian}
Let $F$ be the Fano variety of lines on $X$, a general cubic fourfold containing a plane $P$ over a field $k$ of characteristic not $2$, and let $M$ be the Mukai flop of $F$ along $P^*$, another smooth fourfold. Then there is a commutative diagram
\begin{center}
\begin{tikzcd}
\Bl_{P^*}F \arrow[d] \arrow[r] & M \arrow[d,"\rho"] \\
F \arrow[r,dashed] & \bP^2
\end{tikzcd}
\end{center}
where $\rho$ is fibered in torsors of abelian surfaces. In other words, $F$ admits a rational Lagrangian fibration over $k$.
\end{theo}

For a hyperplane $H\supset P$, the cubic threefold $X\cap H$ contains a plane, and for a general hyperplane $H$, $X\cap H$ is smooth away from $P$, so by Theorem~\ref{theo:fanothree}, one component of its Fano scheme is birational to a torsor $T$ of an abelian surface. These torsors $T$ for various hyperplanes $H\supset P$ are precisely the smooth fibers of $\rho$.

The outline of this paper is as follows: Section~\ref{sect:threefold} gives a detailed analysis of the Fano scheme of lines on a general cubic threefold containing a plane, and Sections~\ref{sect:rationality} and~\ref{sect:lagrangian} contain proofs of Theorems~\ref{theo:rationality} and ~\ref{theo:lagrangian}, respectively.

\bigskip

\noindent {\bf Acknowledgments:} I am very grateful for the help of my advisor, Nicolas Addington, who suggested this project for my Ph.D. thesis and guided much of my early progress. Helpful conversations with Sarah Frei also motivated much of the arithmetic content of the paper. While working on this project, I was partially supported by NSF grants DMS-1902213 and DMS-2039316.

\section{Lines on a cubic threefold containing a plane}\label{sect:threefold}

Let $Y\subset\bP^4$ be a cubic threefold containing a plane $P$. Writing the defining equations for the plane as $x_0=x_1=0$ and for $Y$ as $x_0Q_0+x_1Q_1=0$, the singular locus of $Y$ along $P$ is the intersection of conics
\[
Z:=\{x_0=x_1=Q_0=Q_1=0\}=P\cap \mathrm{Sing}(Y).
\]
Even if $Y$ is smooth away from $P$, $\dim(Z)$ may be positive, but we restrict to the case where $Y$ is {\emph{general}}, by which we mean that $Y$ contains only one plane, and $\dim Z=0$. In Section~\ref{sect:lagrangian}, we show that if $X$ is a smooth cubic fourfold containing exactly one plane $P$, then any section of $X$ by a hyperplane containing $P$ is general, motivating our definition here.

\subsection{The quadric surface fibration}

Let $\cov Y=\Bl_PY$. Projection from $P$ induces a quadric surface fibration $q\colon\cov Y\to\bP^1$ with a sextic discriminant locus $\Delta\subset\bP^1$. Since $Y$ contains only one plane, the degenerate fibers of $q$ are no worse than cones. Geometrically, each fiber of $q$ has either one or two rulings depending on whether the fiber is singular or smooth.

\begin{lemm}\label{lemm:singthreefold}
Let $Y$ be a general cubic threefold containing a plane $P$, let $\cov Y=\Bl_PY$, and let $Z=\mathrm{Sing(Y)\cap P}$. Then 
\begin{enumerate}
    \item[(i)] $Z$ is a complete intersection of two conics in $P$,
    \item[(ii)] there is a bijection between quadric surfaces in $Y$ and conics containing $Z$, given by $Q\mapsto Q\cap P$,
    \item[(iii)] each $k$-point of $Z$ gives a section of the quadric surface fibration $q\colon\cov Y\to\bP^1$,
    \item[(iv)] the singularities of $Y$ are isolated, and
    \item[(v)] if the discriminant of $q$ is reduced, then $\cov Y$ and $Y\setminus P$ are smooth.
\end{enumerate}
\end{lemm}
\begin{proof}
Fix coordinates so that $P=\{x_0=x_1=0\}$, and write the defining equation for $Y$ as $x_0Q_0+x_1Q_1=0$. Since $Y$ is general, $Z$ is zero-dimensional and cut out by the equations
\[
x_0=x_1=Q_0=Q_1=0,
\]
from which we deduce (i).

For (ii), note that each quadric surface in $Y$ is cut out by the equations
\[
sx_1-tx_0=sQ_0+tQ_1=0
\]
for $(s:t)\in\bP^1$, and each point in $Z$ satisfies $x_0=x_1=Q_0=Q_1=0$ so belongs to each quadric surface in the pencil. Hence each quadric surface in $Y$ contains $Z$, and the map $Q\mapsto Q\cap P$ identifies the pencil of quadrics in $X$ with the pencil of conics through $Z$. Moreover, each point in $Z$ gives a section of $q$ since it is contained in each quadric surface in $Y$, proving (iii).

Now we consider the singularities of $\cov Y$. For $y\in \cov Y$, we have 
\[
T_y(\cov Y)=T_yQ+T_{q(y)}\bP^1
\]
where $Q\ni y$ is the quadric surface containing $y$. Hence if $\dim T_y(\cov Y)>3$, then $\dim T_yQ>2$, which happens for at most six points, i.e. the cone points of each of the finitely many singular fibers of $q$. Since the blowup $\cov Y\to Y$ is an isomorphism away from $P$, there are only finitely many singular points on $Y$, proving (iv).

A local computation shows that $\cov Y$ is smooth if and only if the discriminant of $q$ is reduced. When $\cov Y$ is smooth, $Y$ is smooth away from $P$, proving (v).
\end{proof}

Let $\cF=F(\cov Y/\bP^1)$ be the relative Fano variety of lines, parametrizing lines in fibers of $q$. The Stein factorization of the projection $F(\cov Y/\bP^1)\to\bP^1$ consists of a conic bundle $p:\cF\to C$ and a double cover $g:C\to\bP^1$ branched over $\Delta$. The curve $C$ has genus $2$ and is smooth if and only if $\Delta$ is reduced. Note that a line $L\subset Y$ is contained in at most one fiber of $q$, and each line $L\subset Y$ meeting $P$ in exactly one point is contained in a fiber of $q$, so we can embed $\cF$ in the Fano scheme of lines on $Y$ as the closure of the set of lines meeting but not contained in $P$.

Importantly, sections of $p$ are in bijection with sections of $q$: specifying a point in each smooth fiber of $q$ is the same as specifying a line in each ruling of that fiber, explained further in \cite[Sect. 3]{hassetttschinkel}. By Lemma~\ref{lemm:singthreefold}, each (geometric) point $z\in Z$ gives a section of $q$, so let $\tau_z$ denote the corresponding (geometric) section of $p$. In a slight abuse of notation, for $c\in C$ we variously regard $\tau_z(c)$ as a line in $Y$ and as the corresponding point in $\cF$.

\subsection{The Fano scheme of lines}
This section describes the irreducible components of the Fano scheme of lines on $Y$, denoted $F(Y)$. These lines come in three types: those contained in $P$, those meeting $P$ once, and those disjoint from $P$. A line that meets $P$ once belongs to $\cF$, and a line contained in $P$ belongs to the dual plane $P^*$; let $\cU\subset F(Y)$ be the open subsceheme composed of lines in $Y$ not meeting $P$.

\begin{prop}\label{prop:fanothree}
$F(Y)$ contains the following components: 
\begin{enumerate}
    \item the plane $P^*$ dual to $P$,
    \item a ruled surface $\cF$ over a curve $C$ of arithmetic genus $2$,
    \item and a singular surface $\oU$, geometrically birational to $\Sym^2C$.
\end{enumerate}
\end{prop}

\begin{proof}
All that requires proof is the birational geometry of $\oU$ over the separable closure $k^s$.

Choose a point $z\in Z$. For a line $L\subset Y\setminus P$, let $P'$ be the plane spanned by $L$ and $z$. Since $Y$ is singular at $z$,  so too is the curve $P'\cap Y$. Thus $P'\cap Y$ consists of three lines: $L$ and two other lines $M$ and $N$ (which may coincide with one another), both containing $z$. Note $[M],[N]\in\cF$ since $M$ and $N$ meet but are not contained in $P$. The equation $\varphi_z([L])=\{p([M]),p([N])\}$ defines a morphism $\varphi_z:\cU\to\Sym^2C$.

To show $\varphi_z$ is a birational equivalence, we construct its rational inverse. Let $V\subset\Sym^2C$ be the open set consisting of pairs $\{c,d\}$ of distinct points such that the lines $\tau_z(c)$ and $\tau_z(d)$ are not both contained in $P$. For such a pair, the plane $P'$ spanned by $\tau_z(c)$ and $\tau_z(d)$ meets $Y$ in a third line $L$, and we define a morphism $\psi_z:V\to F(Y)$ by $\psi_z(c,d)=[L]$. It is straightforward that $\psi_z\circ\varphi_z=\text{id}_\cU$, as needed.
\end{proof}

\begin{rema}\label{rema:hilb}
Let $S_z\subset Y$ be the union over $c\in C$ of the lines $\tau_z(c)$, which is a cone over $C$. An unreduced length-$2$ subscheme of $C$ specifies a line $\tau_z(c)\subset S_z$ and a normal direction to that line in $S_z$, and together these span a plane $P'$ meeting $Y$ in a third line. In this way, one can instead define $\psi_z$ as a rational map $\Hilb^2C\dashrightarrow\cU$. The domain of $\psi_z$ under this new definition is the complement of at most three points, i.e. the pairs $\{c,d\}$ of distinct points with $\tau_z(c),\tau_z(d)\subset P$. 
\end{rema}

A line $L\subset Y\setminus P$ gives a section of $q$ and hence a section $\sigma_L$ of $p$. When $C$ is smooth, so too is $\cF$, and the following lemma describes the numerics of the two types of sections of $p$ mentioned so far. These calculations are used later in Section~\ref{sect:rationality}.

\begin{lemm}\label{lemm:intpairing}
Suppose $C$ is smooth, and let $[L],[M]\in\cU$ and $z,w\in Z$. Under the intersection pairing on $\cF$, 
\begin{enumerate}
    \item[(i)] $\sigma_L.\tau_z=2$,
    \item[(ii)] $\sigma_L.\sigma_M=3$, and
    \item[(iii)] $\tau_z.\tau_w=1$.
\end{enumerate}
\end{lemm}
\begin{proof}
We may work over $k^s$. The equation (i) follows from the proof of Theorem~\ref{theo:fanothree}: $[M]\in\tau_z(C)$ if and only if $z\in M$, $[M]\in\sigma_L(C)$ if and only if $L\cap M=\emptyset$, and there are two lines meeting $z$ and $L$. For (ii), it is enough to calculate $\sigma_L.\sigma_M$ for a particular choice of $L$ and $M$: indeed, Theorem~\ref{theo:fanothree} implies $\cU$ is irreducible when $C$ is smooth, and a curve in $\cU$ connecting two points gives an algebraic (hence also numerical) equivalence between the corresponding sections of $\cF$. A generic hyperplane section of $Y$ is a smooth cubic surface $S$ and contains skew lines not meeting $P$, which we may assume for the purpose of our calculation are $L$ and $M$.

Note that the span of $L$ and $M$ intersects $P$ in a line $\ell$, and the intersection number $\sigma_L.\sigma_M$ counts how many lines in $Y$ meet $L$, $M$, and $P$, which are exactly the lines in $S$ meeting $L$, $M$, and $\ell$. Given three disjoint lines in a cubic surface, there are three other lines meeting each of those, proving (ii).

Because $\Num(\cF)$ is generated by the class $f$ of a fiber and the class of any section (see, for example, \cite[V.2]{hartshorne}), there are integers $a,b\in\bZ$ with $\tau_z\sim a\sigma_L+bf$. Pairing with $f$ yields $a=1$, and pairing with $\sigma_L$ yields $b=-1$. Thus $\tau_z\sim \sigma_L-f$, and the same is true for $\tau_w$, so (iii) follows from (i) and (ii).
\end{proof}

The main object of study in later sections is the surface $\oU\subset F(Y)$. To understand the boundary of this surface, $\oU\setminus\cU$, we analyze the pairwise intersections of the three components of $F(Y)$.

\begin{lemm}\label{lemm:fcapp}
$\cF\cap P^*$ is zero-dimensional.
\end{lemm}
\begin{proof}
If $[L]\in\cF\cap P^*$, then $L$ is contained in a quadric surface $Q$ meeting $P$ in a degenerate conic (one component of which is $L$). There are at most three degenerate conics containing $Z$, so there are finitely many quadrics $Q$ for which $Q\cap P$ is degenerate by by Lemma~\ref{lemm:singthreefold} (ii). Each of these quadrics contributes two points to $\cF\cap P^*$.
\end{proof}

\begin{rema}
Since $\cF$ and $P^*$ meet in codimension 2, Hartshorne's Connectedness Theorem \cite[Thm. 18.12]{eisenbud} implies that
\[
\overline{F(Y)\setminus\oU}=\cF\cup P^*
\]
is not Cohen-Macaulay. By \cite[Thm. 21.23]{eisenbud}, $\oU$ is not Cohen-Macaulay. In particular, $\oU$ is not smooth.
\end{rema}

It is well-known that $F(Y)$ is smooth at $[L]$ if and only if $Y$ is smooth along $L$ (see, for example, \cite{barth} or \cite[Prop. 6.24]{3264}). The following two lemmas also use the basic fact that a scheme is singular along the intersection of any two irreducible components.

\begin{lemm}\label{lemm:ubarp}
$\oU\cap P^*$ is the union of the pencils $z^*\subset P^*$ of lines in $P$ through each of the points $z\in Z$.
\end{lemm}
\begin{proof}
A line in $P$ is a singular point of $F(Y)$ if and only if it intersects $Z$ nontrivially, so 
\[
P^*\cap\text{Sing}(F(Y))=\bigcup_{z\in Z}z^*.
\]
Moreover, $P^*$ is smooth, so any point in $P^*\cap\text{Sing}(F(Y))$ belongs to a second component of $F(Y)$. Then
\[
\bigcup_{z\in Z}z^*=(P^*\cap\oU)\cup(P^*\cap\cF).
\]
The left-hand side is purely 1-dimensional, and  $P^*\cap\cF$ is 0-dimensional by Lemma~\ref{lemm:fcapp}, so 
\[
\bigcup_{z\in Z}z^*=\oU\cap P^*,
\]
as needed.
\end{proof}

\begin{lemm}\label{lemm:graphpsiz}
Suppose $z\in Z$ is a rational point, so the morphism $\psi_z:\Hilb^2C\dashrightarrow\oU$ from Theorem~\ref{theo:fanothree} and Remark~\ref{rema:hilb} is defined. Let $\Gamma$ be the graph of $\psi_z$. Then
\begin{enumerate} 
\item[(i)] the projection $\pi_2:\Gamma\to\oU$ is an isomorphism except over the zero-dimensional subscheme $\oU\cap\cF\cap P^*$, over which the fibers contain two points,
\item[(ii)] $\Gamma$ is the blowup of $\oU$ along $\oU\cap P^*$,
\item[(iii)] the projection $\pi_1:\Gamma\to\Hilb^2C$ is the blowup of $\Hilb^2C$ at the finitely many points along $E$ corresponding to the pairs $\{c,d\}$ for which $\tau_z(c),\tau_z(d)\in P^*$.
\end{enumerate}
\end{lemm}
\begin{proof}
We assume $Z$ is reduced, though it is straightforward to adapt the argument in each of the cases where $Z$ is nonreduced. For each $c\in C$, let $\bar c$ denote the image of $c$ under the hyperelliptic involution. Let $E=\{\{c,\bar c\}\;|\;c\in C\}\subset\Hilb^2C$, let $s_1,s_2,s_3\in\Hilb^2C$ be the points corresponding to the pairs $\{c,d\}$ for which $\tau_z(c),\tau_z(d)\in P^*$, and let $E_i=\pi_1^{-1}(s_i)$. Note that $\psi_z$ is defined except at $s_1,s_2,s_3$.

First, we claim $\pi_2$ is an isomorphism away from $E_1\sqcup E_2\sqcup E_3$. Since $\psi_z\circ\varphi_z=\mathrm{id}_{\cU}$, we know $\cU$ it contained in the image of $\psi_z$ hence also of $\pi_2$, and so $\pi_2$ is surjective. Moreover, $\psi_z$ is injective on its domain, so $\pi_2$ is injective away from $E_1\sqcup E_2\sqcup E_3$, as needed.

It is straightforward to check that for each point $w\neq z\in Z$, the pencil $w^*$ is not contained in the image of $\psi_z$. Therefore, $\pi_2$ projects $E_1\sqcup E_2\sqcup E_3$ onto the pairwise-intersecting triple of lines 
\[
\bigcup_{w\neq z\in Z}w^*.
\]
Also noticing that $\pi_2(\pi_1^{-1}(E))=z^*$, we obtain (i).

Using the universal property of the blowup, $\pi_2$ factors through a morphism $\rho:\Gamma\to\Bl_{\oU\cap P^*}\oU$. The morphism $\Bl_{\oU\cap P^*}\oU\to\oU$ fails to be an isomorphism over the same six points as $\pi_2$, i.e. the six pairwise intersections of lines in $\oU\cap P^*$ which together compose $\oU\cap\cF\cap P^*$. Since $\pi_2$ is two-to-one over these points, $\rho$ must be an isomorphism.

For (iii), notice that the preimage $E_i$ of each point $s_i$ under $\pi_1$ is isomorphic to $\bP^1$. By the universal property of the blowup, $\pi_1$ factors through a morphism $\Bl_{s_1,s_2,s_3}\Hilb^2C\to\Hilb^2C$ which must be an isomorphism.
\end{proof}

Even if $Z$ has no $k$-point, we deduce the following.

\begin{coro}\label{lemm:blubar}
The blowup $\Bl_{\oU\cap P^*}\oU\to\oU$ is an isomorphism away from the finitely many points $\oU\cap\cF\cap P^*$ and two-to-one over $\oU\cap\cF\cap P^*$.
\end{coro}

Over $k$, $\Bl_{\oU\cap P^*}\oU$ contains a collection of skew lines whose Galois orbits correspond to the Galois orbits of points in $Z$, and the collection of lines on $\Bl_{\oU\cap P^*}\oU$ can be blown down to obtain a surface which is geometrically isomorphic to $\Pic^0_C$.

\begin{defi}\label{defi:t}
Let $T$ be the surface obtained by contracting each line in $\Bl_{\oU\cap P^*}\oU$ lying over $z^*\subset\oU$ for $z\in Z$.
\end{defi}

By Lemma~\ref{lemm:graphpsiz}, when $C$ is smooth, $T$ is geometrically isomorphic to $\Pic^0_C$. Over $k$, we will eventually show that $T$ is a torsor of $\Pic^0_C$.

Finally, we work toward analyzing the intersection $\oU\cap\cF$.

\begin{lemm}\label{lemm:ubarsingf}
The intersection $\oU\cap\mathrm{Sing}(\cF)$ is zero-dimensional.
\end{lemm}
\begin{proof}
We may work over $k^s$ and fix a point $z\in Z$. If $C$ is smooth, then so is $\cF$, and there is nothing to prove. So, assume $C$ is singular. Then $\mathrm{Sing}(\cF)$ consists of a conic over each singular point of $C$. To show that such a conic does not belong to $\oU$, we first show that is contains a dense open set not belonging to the image of the rational map $\psi_z$.

Let $c$ be a singular point of $C$, let $Q=q^{-1}(c)$ be the  quadric lying over $c$, and let $L\subset Q$ be a line not intersecting $Z$. Such a line $L$ is generic in $Q$. The cone point $y\in Q$ does not lie in $P$ or else $Y$ would have a fifth node along $P$. 

The plane $P'$ spanned by $L$ and $z$ intersects $Y$ in two other lines: $\tau_z(c)$ and another line $L'$ containing $z$. Since $\tau_z(c)$ and $L$ intersect at $y\not\in P$, the line $L'$ joins $z$ and $L\cap P$. Hence $L'\subset P$. As $L\cap P\not\in Z$, and $\deg(L'\cap Q\cap P)=2$, the line $L'$ does not contain a length-two subscheme of $Z$ so does not belong to the image of $\tau_z$. It follows that $L$ is not in the image of $\psi_z$. 

Now, we apply Lemma~\ref{lemm:graphpsiz}, which shows that the only curves in $\oU$ without dense open sets lying in the image of $\psi_z$ are the pencils $w^*$ for $z\neq w\in Z$. Since $\mathrm{Sing}(\cF)\not\subset P^*$ and $\mathrm{Sing}(\cF)\cap\mathrm{im}(\psi_z)$ is finite, we arrive at the desired result.
\end{proof}

Let $C_z\subset\cF$ denote the image of $\tau_z$.

\begin{lemm}\label{lemm:ubarf}
$\oU\cap\cF=\cup_{z\in Z}C_z$.
\end{lemm}
\begin{proof}
First, we show $C_z\subset\oU\cap\cF$ for each  point $z\in Z$. Note $C_z\subset\mathrm{Sing}(F(Y))$ since each line in $C_z$ meets $Z$. Hence 
\[
C_z\subset\cF\cap\text{Sing}(F(Y))=(\oU\cap\cF)\cup(\cF\cap P^*)\cup\text{Sing}(\cF).
\]
But $\cF\cap P^*$ is finite by Lemma~\ref{lemm:fcapp}, $\mathrm{Sing}(\cF)$ consists of finitely many conics, and the genus of $C_z$ is $2$, so there is no embedding \[C_z\to(\cF\cap P^*)\cup\mathrm{Sing}(\cF).\] Therefore $C_z\subset\oU\cap\cF$.

For the reverse inclusion, note that if $L\in\oU\cap\cF$, then $L$ is a singular point of $F(Y)$ so passes through a singular point of $Y$. The lines in $\cF$ passing through singular points of $Y$ are parametrized by the curves $C_z$ for $z\in Z$ and by $\mathrm{Sing(\cF)}$, so
\[
\oU\cap\cF\subset\bigcup_{z\in Z}C_z\cup\mathrm{Sing}(\cF).
\]
By Lemma~\ref{lemm:ubarsingf}, $\oU\cap\mathrm{Sing}(\cF)$ is zero-dimensional, so it remains to show that $\oU\cap\cF$ is equidimensional. Indeed, $F(Y)$ is Gorenstein, and $P^*$ is Cohen-Macaulay, so 
\[
\overline{F(Y)\setminus P^*}=\cF\cup\oU
\]
is Cohen-Macaulay by \cite[Thm. 21.23]{eisenbud}. By \cite[Cor. 18.11]{eisenbud}, $\cF\cap\oU$ is purely 1-dimensional.
\end{proof}

\begin{prop}\label{coro:ubar} We have the following description of the boundary of $\oU$:
\[
\oU=\cU\cup\bigcup_{z\in Z}(z^*\cup C_z).
\]
\end{prop}
\begin{proof}
This follows directly from Lemmas~\ref{lemm:ubarp} and \ref{lemm:ubarf}.
\end{proof}

\begin{coro}\label{coro:tpoints}
We can identify the surface $T$ from Definition~\ref{defi:t} with
\[
T=\left(\oU\setminus\bigcup_{z\in Z}z^*\right)\cup Z.
\]
\end{coro}

\subsection{A collection of rational involutions}

Recall that when $C$ is smooth, $T$ is geometrically an abelian surface. In the next subsection, we equip $T$ with the structure of a $\Pic^0_C$-torsor; to define an action by $\mathrm{Div}(C)$ on $T$, we use a collection of involutions on $T$ defined in this section. First, we define a collection of rational involutions on $\oU$.

\begin{defi}\label{defi:ratinv}
For $c\in C$, define a rational map $i_c:\oU\dashrightarrow F(Y)$ as follows: $c$ specifies a ruling in a quadric surface in $Y$, and a line $L\subset Y\setminus P$ meets a unique line $M$ in this ruling. The plane spanned by $L$ and $M$ meets $Y$ in a third line, $i_c(L)$.
\end{defi}

We have defined $i_c$ on $\cU$, though the domain turns out to be larger, as explained shortly.

\begin{lemm}\label{lemm:psi}
For $c,d\in C$ and $z\in Z$, \[i_c(\tau_z(d))=i_d(\tau_z(c))\] whenever both sides of the equation are defined.
\end{lemm}
\begin{proof}
It suffices to check that the two rational maps $C\times C\to\oU$ defined by sending $(c,d)$ to $i_c(\tau_z(d))$ and to $i_d(\tau_z(c))$ agree on a dense open set. For the open set, take
\[
V=\{(c,d)\;|\;\bar c\neq d~\mathrm{and}~\tau_z(c),\tau_z(d)\not\subset P \}\subset C\times C.
\]
For $(c,d)\in V$, the lines $\tau_z(c)$ and $\tau_z(d)$ meet at $z$, and the third line in the plane spanned by $\tau_z(c)$ and $\tau_z(d))$ is, by definition, both $i_c(\tau_z(d))$ and $i_d(\tau_z(c))$.
\end{proof}

\begin{rema}\label{rema:psi}
Lemma~\ref{lemm:psi} affords another description of the birational maps $\psi_z\colon\Hilb^2C\dashrightarrow\oU$ from the proof of Theorem~\ref{theo:fanothree}: \[\psi_z(\{c,d\})=i_c(\tau_z(d)).\]
\end{rema}

\begin{lemm}
Each $i_c$ is a rational involution whose image is contained in $\oU$.
\end{lemm}
\begin{proof}
It is easy to see that $i_c^2=\mathrm{id}_{\oU}$. To show that the image of $i_c$ is contained in $\oU$, it is enough to show that the image contains $\cU$. 

Let $[L]\in\cU$ and take $[M]\in\cU$ be any line with $L\cap M\neq\emptyset$. The plane $P'$ spanned by $L$ and $M$ intersects $P$ at a single point $y$ and intersects $Y$ in the union of $L$, $M$, and a third line $N$. The point $y\in P$ must lie on $N$, so $[N]\in\cF$. Taking $c\in C$ to be the ruling $N$ belongs to, $[L]=i_c([M])$.
\end{proof}

\begin{lemm}\label{lemm:domainic}
For all $c$, the domain of $i_c$ includes the following:
\begin{enumerate}
    \item[(i)] the subscheme $\cU$;
    \item[(ii)] $\tau_z(d)$ for $z\in Z$ and $d\in C$ if $d\not\in\{c,\bar c\}$ and $\tau_z(d)\not\subset P$;
    \item[(iii)] $\tau_z(c)$.
\end{enumerate}
If also $\tau_z(c)\not\subset P$ for all $z$, then the domain of $i_c$ includes
\begin{enumerate}
    \item[(iv)] any line $\ell\in P^*\cap\oU$ not of the form $\tau_z(d)$;
    \item[(v)] $\tau_z(\bar c)$.
\end{enumerate}
\end{lemm}
\begin{proof}
Let $Q\subset Y$ be the quadric surface with ruling $c$. It is clear that $i_c$ is defined on $\cU$ and on $\tau_z(d)$ meeting the conditions in (ii): each of these lines $L$ meets a unique line $M$ in the ruling $c$ of $Q$, and the plane $P'$ spanned by $L$ and $M$ is not $P$ so meets $Y$ in dimension one, giving a third line in $\oU$. For (iii), the same argument holds, but the plane $P'$ is spanned by $\tau_z(d)$ and a normal vector to that line in the surface
\[
\bigcup_{d\in C}\tau_z(d)\subset Y.
\]
as in Remark~\ref{rema:hilb}. A line $\ell$ of the form in (iv) also meets a unique line in the ruling $c$ of $Q$, namely $\tau_z(c)$. If $\tau_z(c)\not\subset P$, then the plane spanned by $\ell$ and $\tau_z(c)$ is not $P$ so meets $Y$ in dimension one, giving a third line in $\oU$. For (v), note that $\psi_z(\{c,\bar c\})$ is well-defined; it is the line $[T_zQ\cap P]$. Using $i_c(\tau_z(\bar c))=\psi_z(\{c,\bar c\})$, we see $\tau_z(\bar c)$ is in the domain of $i_c$.
\end{proof}

\begin{defi}
For $c\in C$, let $j_c:T\dashrightarrow T$ be the rational map already defined on $\cU\subset T$ by the rational involution $i_c$.
\end{defi}

\begin{lemm}\label{lemm:jc}
Each rational map $j_c$ extends to a morphism.
\end{lemm}

When $C$ is smooth, Lemma~\ref{lemm:jc} can been deduced from the fact that any rational map from a smooth variety to an abelian variety extends to a morphism. However, it is useful to have descriptions for how $j_c$ is defined at each point on $T$, motivating a more explicit proof.

\begin{proof}[Proof of Lemma~\ref{lemm:jc}]
Let $Q$ be the quadric surface with a ruling specified by $c$. The following cases describe all types of points on $T$:
\begin{enumerate}
    \item[(i)] $[L]\in\cU$;
    \item[(ii)] $\tau_z(d)$ for $z\in Z$ if $\bar c\neq d\in C$ and $\tau_z(d)\not\subset P$;
    \item[(iii)] $\tau_z(\bar c)$ for $z\in Z$ if $\tau_z(\bar c)\not\subset P$;
    \item[(iv)] $z\in Z$.
\end{enumerate}
On each of these types of points, $j_c$ is defined as follows.
\begin{enumerate}
    \item[(i)] Note $i_c([L])\not\in P^*$, so $i_c([L])$ can be identified with a point of $T$, and $j_c([L])=i_c([L])$.
    \item[(ii)] The line in the $c$-ruling of $Q$ that $\tau_z(d)$ meets is $\tau_z(c)$. By Lemma~\ref{lemm:domainic}, $i_c(\tau_z(d))$ is defined: it is the third line $\ell$ in the plane spanned by $\tau_z(c)$ and $\tau_z(d)$. We can identify $j_c(\tau_z(d))=i_c(\tau_z(d))$ so long as the line $i_c(\tau_z(d))$ does not lie in $P$. To verify this, note that the rational map $\psi_z:\Sym^2C\dashrightarrow\oU$ sending $\{a,b\}$ to $i_a(\tau_z(b))$ is injective, and it is straightforward to check that the pencil of lines in $P$ through $z$ is the image of the pencil of pairs $\{a,\bar a\}$. Having assumed $\bar c\neq d$, we obtain $i_c(\tau_z(d))\not\subset P$.
    \item[(iii)] Suppose $\tau_z(c)\not\subset P$. Then by Lemma~\ref{lemm:domainic}, $i_c(\tau_z(\bar c))\in z^*$. Since $z^*$ is contracted in $T$ to a point identified with $z$, $j_c(\tau_z(\bar c))=z$.

    If $\tau_z(c)\subset P$, then $i_c(\tau_z(\bar c))$ is not defined: $\tau_z(\bar c)$ meets each line in the $c$-ruling of $Q$. Given a line $L\subset Q$ in the $c$-ruling, the plane $P'$ spanned by $L$ and $\tau_z(c)$ intersects $Y$ in a third line $M$. Noting that $P'$ contains $z$ and $L\cap P$, we see $[M]\in z^*$. Different choices of $L$ give different lines $[M]\in z^*$, but $z^*$ is contracted to $z$ in $T$, so we may define $j_c(\tau_z(\bar c))=z$.
    \item[(iv)] Since $j_c$ is an involution, applying $j_c$ to both sides of (iii) gives $j_c(z)=\tau_z(\bar c)$.
\end{enumerate}
\end{proof}

\subsection{Arithmetic of the torsor}
Here, we restrict to the case that $C$ is smooth, so $T$ is geometrically isomorphic to the abelian surface $\Pic^0_C$. This section describes how to endow $T$ with the structure of a $\Pic^0_C$-torsor, following the approach of Wang in \cite{wang}. The strategy is to put a group scheme structure on the disconnected group variety 
\[
{\Pic^0_C}\sqcup {T}\sqcup{\Pic^1_C}\sqcup {T'}
\]
where $T'$ is a variety isomorphic to $T$ whose points are written $-t$ for each $t\in T$.

The proofs of Lemmas~\ref{lemm:invcommute} and \ref{lemm:freelytransitive} below use information about the incidence relations between lines on cubic surfaces. We fix the following notation for the lines on a smooth cubic surface over an algebraically closed field, similar to the notational scheme used in \cite[V.4]{hartshorne}.

\begin{notation} A smooth cubic surface $S$ over an algebraically closed field is the blowup of $\bP^2$ at six points $p_1,\dots,p_6$ in general position and contains 27 lines:
\begin{itemize}
    \item[(i)] the six exceptional divisors $E_i$;
    \item[(ii)] for each pair $\{i,j\}$, the proper transform $\ell_{ij}$ of the line joining $p_i$ to $p_j$;
    \item[(iii)] and for each $i$, the proper transform $F_i$ of the conic passing through all the $p_j$ except $p_i$.
\end{itemize}
The incidence relations are as follows:
\begin{itemize}
    \item[(i)] $E_i$ meets $\ell_{ij}$ and $f_j$ for $i\neq j$;
    \item[(ii)] $\ell_{ij}$ meets $E_i$, $E_j$, $F_i$, $F_j$, and $\ell_{hk}$ for $\{i,j\}\cap\{h,k\}=\emptyset$;
    \item[(iii)] $F_i$ meets $E_j$ for $i\neq j$ and $\ell_{ij}$ for all $j$.
\end{itemize}
Moreover, given six pairwise skew lines $L_1,\dots,L_6\subset S$, the blowdown of the $L_i$ is isomorphic to $\bP^2$, so one may assume $L_i=E_i$.
\end{notation}

\begin{lemm}\label{lemm:invcommute}
The rules
\[
t+(c)=-j_{\bar c}(t)
\]
and 
\[
-t+(c)=j_c(t)
\]
define an action of $\mathrm{Div}(C)$ on $T\sqcup T'$.
\end{lemm}
\begin{proof}
For the action to be well-defined, $(c)+(d)$ must act the same as $(d)+(c)$, so we need to check $j_c\circ j_{\bar d}=j_d\circ j_{\bar{c}}$ for any $c,d\in C$.

Work over $k^s$. As $T$ is irreducible, it suffices to show that the two maps $C\times C\times T\to T$ sending $(c,d,t)$ to $j_d\circ j_{\bar c}(t)$ and $j_c\circ j_{\bar d}(t)$ agree on a nonempty open set. In particular, we can choose $t=[L]\in\cU$ and $c\neq d\in C$ such that the $\bP^3$ spanned by the lines $\sigma_L(\bar c)$ and $\sigma_L(\bar d)$ meets $Y$ in a smooth cubic surface S. 

Label $\sigma_L(\bar c)=E_1$, $\sigma_L(\bar d)=E_2$, and $L=\ell_{12}$, using the notation for the lines on a smooth cubic surface. Notice that $S\cap P$ is a line meeting $\sigma_L(\bar c)$ and $\sigma_L(\bar d)$ but not $L$. Label this line $F_6$.

Now we calculate $j_{\bar c}([L])=[F_2]$: indeed, only  the line $F_2$ meets $E_1$ and $\ell_{12}$. The line $\ell_{26}$ meets $E_2$ and $F_6$ hence $\sigma_L(\bar d)$ and $P$, so it is a line in the ruling of a quadric surface specified by $d$. As $j_{\bar c}([L])=[F_2]$ meets $\ell_{26}$, $j_d\circ j_{\bar c}([L])$ is the line meeting $F_2$ and $\ell_{26}$, namely $E_6$. One calculates $j_c\circ j_{\bar d}([L])=[E_6]$ similarly.
\end{proof}

\begin{prop}
The principal divisors act trivially on $T\sqcup T'$, so the action of $\mathrm{Div}(C)$ on $T\sqcup T'$ descends to an action by $\Pic_C$.
\end{prop}
\begin{proof}
First, notice that $T$ and $T'$ are in different orbits of the action by $\text{Div}^0(C)$, and it suffices to show that principal divisors act trivially on $T$.

Every nontrivial divisor class in $\Pic^0_C$ can be represented as a difference of two effective divisors of degree 1 in exactly two ways: if $(c)-(d)$ represents $D$, then so does $(\bar d)-(\bar c)$. By Lemma~\ref{lemm:invcommute}, $j_c\circ j_d=j_{\bar d}\circ j_{\bar c}$, so there is a well-defined morphism $\Pic^0_C\to\text{Aut}(T)$ sending $(c)-(d)$ to $j_{\bar d}\circ j_{\bar c}$.

The map is automatically a homomorphism: its image is a commutative, projective subgroup scheme, and any unital morphism of abelian varieties is a homomorphism. Moreover, $j_{\bar d}\circ j_{\bar c}(t)=t+(c)-(d)$, so the homomorphism $\text{Div}^0(C)\to\text{Aut}(T)$ coming from the group action described earlier factors through the morphism $\Pic^0_C\to\text{Aut}(T)$, i.e. principal divisors are in the kernel.
\end{proof}

Since 
\[
t+(c)+(\bar c)=j_{\bar c}^2(t)=t,
\]
$\omega_C$ acts trivially on $T\sqcup T'$, so the action by $\Pic_C$ descends to an action by the quotient
\[
\Pic_C/\omega_C\cong{\Pic^0_C}\sqcup{\Pic^1_C}.
\]
Note that $\Pic^0_C$ acts on each of $T$ and $T'$, and an element of $\Pic^1_C$ exchanges the components $T$ and $T'$.

\begin{prop}\label{lemm:freelytransitive}
${\Pic^0_C}\sqcup{\Pic^1_C}$ acts simply transitively on $T\sqcup T'$.
\end{prop}
\begin{proof}
To show that the action of ${\Pic_C^0}\sqcup{\Pic_C^1}$ on $T\sqcup T'$ is transitive, we must check that for each $s,t\in T$,
\begin{enumerate}
\item[(i)] there is a divisor class $[D]\in\Pic^1_C$ so that $s+[D]=-t$;
\item[(ii)] there is a divisor class $[E]\in\Pic^0_C$ so that $s+[E]=t$.
\end{enumerate}
Moreover, we verify
\begin{enumerate}
\item[(iii)] the action of $\Pic^0_C$ on $T$ is free;
\item[(iv)] the action of ${\Pic^0_C}\sqcup{\Pic^1_C}$ on $T\sqcup T'$ is free.
\end{enumerate}
(i) First, we show that when $[L],[M]\in\cU$ span a $\bP^3$ meeting $Y$ in a smooth cubic surface, there exists $[D]\in\Pic^1_C$ so that $[L]+[D]=-[M]$.

Using Lemma~\ref{lemm:intpairing}, let $c,d,e\in C$ be the points lying below $\sigma_L\cap\sigma_M\subset\cF$. No two of $\sigma_L(c)$, $\sigma_L(d)$, $\sigma_L(e)$ meet: if $\sigma_L(c)$ met $\sigma_L(d)$, then they would span a plane containing $L$ and $M$ which would not meet $X$ in degree three. Working over $k^s$ and using the standard notation for lines on the cubic surface $S$, assume $L=E_1$, $M=E_2$, $S\cap P=E_3$, $\sigma_L(c)=F_4$, $\sigma_L(d)=F_5$, and $\sigma_L(e)=F_6$. One calculates 
\[
j_{\bar d}\circ j_c([L])=j_{\bar d}([\ell_{14}])=[\ell_{26}]=j_e([M]),
\]
so 
\[
(e)-[M]=[L]-(c)-(d),
\]
and 
\[
[L]+(\bar c)+(\bar d)+(\bar e)=-[M].
\]
Taking $[D]=(\bar c)+(\bar d)+(\bar e)-\omega_C$, we are done. 

Now, for $[L]\in\cU$, define a morphism $s_L:\Pic^1_C\to T'$ by $s_L([D])=[L]+[D]$. By the above, the image of $s_L$ contains the dense open set
\[
V=\{-[M]\;|\; [M]\in\cU\;\text{and}\;\text{span}(L,M)\cap Y\;\text{is a smooth cubic surface}\},
\]
so $s_L$ is surjective.

Next, consider the summation map 
\[
\Sigma:T\times\Pic^1_C\to T',\;\;(u,[D])\mapsto u+[D].
\]
Let $-t\in T'$. As $s_L$ is surjective for each $[L]\in\cU$, we get 
\[
\cU\subset\pi_1(\Sigma^{-1}(-t))
\]
where $\pi_1$ is the projection $T\times\Pic_C^1\to T$. Since $\pi_1$ is closed, 
\[
\pi_1(\Sigma^{-1}(-t))=T.
\]
That is, for each $s\in T$ there exists $[D]\in\Pic^1_C$ so that $s+[D]=-t$.

(ii) Apply (i) to obtain $[D],[E]\in\Pic^1_C$ with $s+[D]=-t$ and $t+[E]=-t$. Then $s+[D-E]=t$.

(iii) Work over $k^s$. It is enough to check that if $(c)+(d)$ fixes a point $z\in Z\subset T$, then $(c)+(d)=\omega_C$. Indeed, if $z= z+(c)+(d)$, then
\begin{align*}
    z+(\bar d)&=z+(c)\\
    j_d(z)&=j_{\bar c}(z)\\
    \tau_z(\bar d)&=\tau_z(c),
\end{align*}
so $\bar d =c$ since $\tau_z$ is an embedding.

(iv) The argument for (iii) also shows $\Pic^0_C$ acts freely on $T'$. Moreover, if $t+[D]=t+[E]$ for $t\in T\sqcup T'$ and $[D],[E]\in\Pic^1_C$, then $[D]-[E]\in\Pic^0_C$ fixes $t$, so  $[D]-[E]=0$ by (iii).
\end{proof}

We now prove our main structural results about $T$.

\begin{prop}
There is a commutative group law on 
\[
G={\Pic^0_C}\sqcup T\sqcup{\Pic^1_C}\sqcup T'
\]
extending the group law on ${\Pic^0_C}\sqcup{\Pic^1_C}$ such that for $s,t\in T\sqcup T'$ and $c\in C$,
\begin{enumerate}
    \item $s+(c)=-j_{\bar c}(s)$ and $-s+(c)=j_c(s)$;
    \item $s+t$ is the unique divisor class $[D]$ such that $-s+[D]=t$ under the action of ${\Pic^0_C}\sqcup{\Pic^1_C}$ on $T\sqcup T'$ defined prior.
\end{enumerate}
\end{prop}
\begin{proof}
All that remains to check is that the group law is associative, i.e. the following all hold for for $s,t,u\in T\sqcup T'$ and $[D],[E],[F]\in{\Pic^0_C}\sqcup{\Pic^1_C}$.
\begin{enumerate}
\item[(i)]    $([D]+[E])+[F]=[D]+([E]+[F])$
\item[(ii)]    $(s+[D])+[E]=s+([D]+[E])$
\item[(iii)]    $(s+t)+[D]=s+(t+[D])$
\item[(iv)]    $(s+t)+u=s+(t+u)$
\end{enumerate}
The first is inherited from the associativity of ${\Pic^0_C}\sqcup{\Pic^1_C}$, and the second follows from the fact that $\text{Div}(C)$ acts on $T\sqcup T'$. For (iii), let $[E]=s+t$, meaning $[E]$ is the unique divisor class such that $-s+[E]=t$. Using (ii),
\[
-s+([E]+[D])=t+[D],
\]
i.e.
\[
-s+((s+t)+[D])=t+[D].
\]
Similarly, $s+(t+[D])$ is defined by the equation
\[
-s+(s+(t+[D]))=t+[D],
\]
so 
\[
-s+((s+t)+[D])=-s+(s+(t+[D])).
\]
Since ${\Pic^0_C}\sqcup{\Pic^1_C}$ acts freely on $T\sqcup T'$,  we deduce
\[
(s+t)+[D]=s+(t+[D])
\]
proving (iii). For (iv), let $K_1=(s+t)+u$ and $K_2=s+(t+u)$. Using (iii) and commutativity,
\begin{align*}
    t+K_1&=t+(u+(s+t))\\
    &=(t+u)+(s+t)\\
    &=(t+s)+(t+u)\\
    &=t+(s+(t+u))\\
    &=t+K_2,
\end{align*}
so again using (iii), the divisor class $t+K_1$ sends $-K_1$ and $-K_2$ both to $t$. As ${\Pic^0_C}\sqcup{\Pic^1_C}$ acts invertibly, $K_1=K_2$.
\end{proof}

\begin{coro}\label{coro:wcrelations}
In $H^1(k,\Pic^0_{C})$, $4[T]=0$ and $2[T]=[\Pic^1_C]$.
\end{coro}
\begin{proof}
There is a short exact sequence of group schemes
\[
0\longrightarrow\Pic^0_C\longrightarrow G\overset\pi\longrightarrow\bZ/4\bZ\longrightarrow0
\]
where $\pi(T)=1$, $\pi(\Pic_C^1)=2$, and $\pi(T')=3$. This yields a short exact sequence 
\[
0\longrightarrow\Pic^0_C(k^s)\longrightarrow G(k^s)\overset\pi\longrightarrow\bZ/4\bZ\longrightarrow0
\]
of Galois modules. In the long exact sequence in Galois cohomology, the connecting homomorphism $H^0(k,\bZ/4\bZ)\to H^1(k,\Pic^0_{C})$ sends $1\mapsto[\pi^{-1}(1)]=[T]$ and $2\mapsto[\Pic^1_C]$.
\end{proof}

This result, together with Proposition~\ref{prop:fanothree}, proves Theorem~\ref{theo:fanothree}.

\section{Rationality of cubic threefolds containing a plane}\label{sect:rationality}

Let $Y\subset\bP^4$ be a general cubic threefold containing a plane $P$ over a field of characteristic not $2$, and suppose further that $Y\setminus P$ is smooth. As in Section~\ref{sect:threefold}, let $\cov Y=\Bl_PY$, let $C$ be the double cover of $\bP^1$ branched over the discriminant of the quadric surface fibration $q\colon\cov Y\to\bP^1$. Since we have assumed $Y\setminus P$ is smooth, so too is $C$ by Lemma~\ref{lemm:singthreefold}. Let $T$ be the $\Pic^0_C$-torsor birational to a the Fano scheme of lines in $Y\setminus P$.

There are two classical rationality constructions for $Y$ over $k^s$: projection from a node on $Y$ or projection onto $L\times P$ where $L\subset Y\setminus P$ is a line. However, neither construction may be available over $k$: by Corollary~\ref{coro:tpoints}, the existence of a $k$-rational node on $Y$ or a $k$-rational line in $Y\setminus P$ is equivalent to the existence of a $k$-point on the torsor $T$ associated to $Y$. If $k$ is finite, Lang's theorem affords the following:

\begin{prop}
Let $Y$ be a general cubic threefold containing a plane $P$ over a finite field $k$ of characteristic not $2$, and suppose $Y\setminus P$ is smooth. Then $Y$ is rational over $k$.
\end{prop}

More generally, we prove that the torsor $T$ completely controls the rationality of $Y$ over $k$:

\begin{prop}\label{prop:rationality}
Let $Y$ be a general cubic threefold containing a plane $P$ over a field $k$ of characteristic not $2$, and suppose $Y\setminus P$ is smooth. Then $Y$ is rational over $k$ if and only if $T(k)\neq\emptyset$.
\end{prop}

In light of the discussion above, proving Theorem~\ref{theo:rationality} amounts to proving Proposition~\ref{prop:rationality}, which we do using intermediate Jacobian torsor obstructions developed by Hassett and Tschinkel in \cite{hassetttschinkelrationality} and by Benoist and Wittenberg in \cite{benoistwitt}.

\subsection{Intermediate Jacobian torsor obstructions}

Here, we give a brief account of the intermediate Jacobian of a threefold as defined in \cite{benoistwitt}, outlining the results used to prove Proposition~\ref{prop:rationality}.

\begin{theo}[\cite{benoistwitt} Theorem 3.1)]
Let $X$ be a smooth, projective, geometrically rational threefold over a field $k$. Then
\begin{enumerate}
    \item[(i)] there is a group scheme $\CH^2_X$ over $k$ whose $k$-points are Galois-invariant rational equivalence classes of codimension-2 cycles on $X_{\bar k}$,
    \item[(ii)] the identity component $(\CH^2_X)^0\subset\CH^2_X$ parametrizing algebraically trivial classes is an abelian variety over $k$, and
    \item[(iii)] the set of $\bar k$-points of the group scheme $\NS^2_X=\CH^2_X/(\CH^2_X)^0$ is isomorphic as a Galois module to $\NS^2(X_{\bar k})$, the set of algebraic equivalence classes of codimension-2 cycles on $X_{\bar k}$.
\end{enumerate}
\end{theo}

For each $\alpha\in\NS^2_X(k)$, the preimage of $\alpha$ under the surjection $\CH^2_X\to\NS^2_X$ is a torsor of $(\CH^2_X)^0$ denoted $(\CH^2_X)^\alpha$.  The arithmetic of these torsors is strictly controlled if $X$ is $k$-rational:

\begin{theo}[\cite{benoistwitt} Theorem 3.11]\label{theo:bw311}
Let $X$ be as above, and suppose $X$ is $k$-rational. Then
\begin{enumerate}
    \item[(i)] there exists a smooth, projective curve $C$ over $k$ and an isomorphism $(\CH^2_X)^0\simeq\Pic^0_C$, and
    \item[(ii)] if $g(C)\ge2$, then for each $\alpha\in\NS^2_X(k)$ there is some $e_\alpha\in\bZ$ for which $[(\CH^2_X)^\alpha]=[\Pic^{e_\alpha}_C]$ in $H^1(k,\Pic^0_C)$.
\end{enumerate}
\end{theo}

Both theorems presented above have stronger statements in \cite{benoistwitt}, but we only use the above versions. Any birational equivalence between $X$ and $\bP^3$ factors as a sequence of blowups along closed points and along regular curves, and the curve $C$ in Theorem~\ref{theo:bw311} is the union of the blown up curves, as long as they are all smooth (for example, if $k$ is perfect). Indeed, if $X_1$ is the blowup of $X$ along a smooth curve $Z$, then
\[
\CH^2_{X_1}\simeq\CH^2_X\times\Pic_Z
\]
whereas if $X_2$ is the blowup of $X$ at a point, then
\[
\CH^2_{X_2}\simeq\CH^2_X\times\bZ
\]
by \cite[Thm. 3.10]{benoistwitt}.

\subsection{Rationality of cubic threefolds containing a plane}

We return to a general cubic threefold $Y\subset\bP^4$ containing a plane $P$ which is smooth away from $P$. The threefold $\cov Y=\Bl_PY$ is smooth by Lemma~\ref{lemm:singthreefold}, and we will use intermediate Jacobian torsor obstructions on $\cov Y$ to study rationality of $Y$.

\begin{lemm}\label{lemm:ch2}
Let $q:X\to\bP^1$ be a quadric surface fibration with reduced discriminant $\Delta$ whose fibers are no worse than cones. Let $D$ be the double cover of $\bP^1$ branched over $\Delta$ parametrizing rulings in fibers of $q$. Then $(\CH^2_X)^0\simeq\Pic^0_D$. 
\end{lemm}
\begin{proof}
For each $c\in D$ let 
\[
[M_d]\in\CH^2(X_{\bar k})^{\Gal(\bar k/k)}=\CH^2_X(k)
\]
be the rational equivalence class of a line in the corresponding ruling of a fiber of $q$. Then there is an embedding
\[
\Pic^0_D\to(\CH^2_X)^0,\;\;\sum_{i}a_i[d_i]\mapsto\sum_{i}a_i[M_{d_i}].
\]
To show that this embedding is an isomorphism, it remains to check that $\dim(\CH^2_X)^0=\dim\Pic^0_D$, for which we may extend scalars to $\bar k$ and fix a section $\sigma:\bP^1\to X$ of $q$. Generalizing the fact that the blowup of a quadric surface at a point is isomorphic to the blowup of $\bP^2$ at a pair of points (which coincide if the quadric is singular), we obtain 
\[
\Bl_D(\bP^1\times\bP^2)\simeq\Bl_{\sigma(\bP^1)}X.
\]
Applying the blowup formula \cite[Prop. 3.10]{benoistwitt} yields 
\[
\CH^2_{\bP^1\times\bP^2}\times\Pic_D\simeq\CH^2_X\times\Pic_{\bP^1},
\]
from which we deduce $\dim\CH^2_X=\dim\Pic^0_D$, as needed.
\end{proof}

In our situation, we obtain $\CH^2_{\cov Y}\simeq\Pic^0_C$.

\begin{proof}[Proof of Proposition~\ref{prop:rationality}]
If $T(k)\neq\emptyset$, we have already seen $\cov Y$ is $k$-rational. So, suppose $\cov Y$ is $k$-rational, and let $\gamma\in\NS^2_{\cov Y}(k)$ be the algebraic equivalence class of a line in $Y\setminus P$. Any two lines in $Y\setminus P$ are rationally inequivalent on $\cov Y$, for a rational equivalence between two lines determines a morphism $\bP^1\to \cU\subset T$, which must be constant since $T$ is geometrically an abelian surface. Hence there is an embedding $\cU\to(\CH^2_{\cov Y})^\gamma$, where $\cU\subset F(Y)$ parametrizes the lines in $Y\setminus P$. As $\cU$ is an open subscheme of $T$, and any rational map between abelian varieties (or their torsors) extends to a morphism, $T$ embeds in $(\CH^2_{\cov Y})^\gamma$. By Lemma~\ref{lemm:ch2}, $\dim((\CH^2_{\cov Y})^\gamma)=2$, so this embedding is an isomorphism.

From Lemma~\ref{lemm:ch2}, $(\CH^2_{\cov Y})^0\simeq\Pic^0_C$, and under the assumption that $\cov Y$ is rational, Theorem~\ref{theo:bw311} affords some $e\in\bZ$ for which
\[
[T]=[(\CH^2_{\cov Y})^\gamma]=[\Pic^e_C]\in H^1(k,\Pic^0_C).
\]
Since $C$ is hyperelliptic, $2[\Pic^e_C]=0$ for all $e$, and applying Corollary~\ref{coro:wcrelations}, 
\[
[\Pic^1_C]=2[T]=2[\Pic^e_C]=0\in H^1(k,\Pic^0_C).
\]
Hence $[\Pic^i_C]=0$ for all $i$, and in particular $[T]=0$, which means $T$ has a $k$-point.
\end{proof}

\begin{coro}
If $0\neq[\Pic^1_C]\in H^1(k,\Pic^0_C)$, then $\cov Y$ is $k$-irrational.
\end{coro}
\begin{proof}
If $\Pic^1_C$ is nontrivial, then so is $T$ by Corollary~\ref{coro:wcrelations}, and we apply Proposition~\ref{prop:rationality}.
\end{proof}

\begin{coro}\label{coro:nosection}
If $Y$ contains a quadric surface $Q$ over $k\subset\bC$ for which $Q(k)=\emptyset$, then $\cov Y$ is irrational over $k$.
\end{coro}
\begin{proof}
Since $Q(k)=\emptyset$, there can be no section of the quadric surface fibration $q$, so also $T(k)=\emptyset$, and we apply Proposition~\ref{prop:rationality}.
\end{proof}

\section{Fibrations from Fano varieties of cubic fourfolds containing a plane}\label{sect:lagrangian}

As mentioned in the introduction, if $X$ is a smooth complex cubic fourfold, then its Fano variety of lines $F$ is a hyperk\"ahler fourfold of K3$^{[2]}$ type. \cite{beauville}. As explained below, $F$ rarely admits any interesting fiber structures, but when $X$ contains a plane, $F$ is birational to another hyperk\"ahler fourfold admitting a Lagrangian fibration. We describe this Lagrangian fibration explicitly and find that the construction works over an arbitrary field of characteristic different from $2$.

\subsection{Hyperk\"ahler motivation}
First, we outline some hyperk\"ahler geometry motivating the construction later in this section. If $M$ is a hyperk\"ahler $2n$-fold, $B$ is a smooth variety, and $f:M\to B$ is a nontrivial surjective morphism with connected, positive-dimensional fibers, then the following must be true:
\begin{enumerate}
    \item[(i)] $B\simeq\bP^n$, proved in \cite{hwang}, and
    \item[(ii)] the restriction of $\sigma$ to any fiber of $f$ is zero, and the fibers of $f$ are abelian varieties, proved in \cite{matsushitafiber}.
\end{enumerate}
Such morphisms are called Lagrangian fibrations. If $M$ is birational to another hyperk\"ahler variety $N$ admitting a Lagrangian fibration, then the composition $M\dashrightarrow N\to\bP^n$ is called a rational Lagrangian fibration. A hyperk\"ahler manifold $M$ of K3$^{[n]}$-type admits a rational Lagrangian fibration if and only if there is some nontrivial class $D\in H^2(M,\bZ)$ belonging to the birational K\"ahler cone and isotropic under the Beauville-Bogomolov form \cite{matsushitadivisor}. By results from \cite[Sect. 6]{markmantorelli}, the existence of any nontrivial isotropic integral divisor class implies the existence of one in the birational K\"ahler cone (for further discussion, see \cite[Cor. 7.3]{mongardi}).

Let $X\subset\bP^5$ be a smooth, complex cubic fourfold, and let $F$ be the Fano variety of lines on $X$, a hyperk\"ahler fourfold on $K3^{[2]}$-type. The Abel-Jacobi map $\alpha\colon H^4(X,\bZ)\to H^2(F,\bZ)$ is compatible with the quadratic forms on each group (namely, the intersection form and Beauville-Bogomolov form) after restriction to primitive cohomology in the sense that $x.y=-q_F(\alpha(x),\alpha(y))$ for $x,y\in H^4(X,\bZ)_{\mathrm{prim}}$. Using this compatibility, it is not difficult to characterize when $X$ is birational to another hyperk\"ahler fourfold $M$ admitting a Lagrangian fibration, as briefly explained below.

Under a framework outlined by Hassett in \cite{hassett}, a cubic fourfold $X$ is called \emph{special} if $\rank H^{2,2}(X,\bZ)>1$, a \emph{labeling} of a special cubic fourfold $X$ is a choice of rank-$2$ primitive sublattice $K\subset H^{2,2}(X,\bZ)$ containing the square of the hyperplane class, and the \emph{discriminant} of $X$ is the discriminant of the intersection form on $K$. The moduli space $\cC_d$ of cubic fourfolds of discriminant $d$ is an irreducible divisor in $\cC$ when $d\ge8$ and $d=0,2\mod 6$. Otherwise, $\cC_d$ is empty. As an example, the space $\cC_8$ parametrizes cubic fourfolds containing at least one plane.

The following fact appears in \cite[Cor. 5.24]{huybrechtscubic}, characterizing when $F$ admits a rational Lagrangian fibration. For a complete proof, see \cite[Prop. 1.4.2]{brookethesis}

\begin{prop}
Let $X\subset\bP^5$ be a smooth, complex cubic fourfold and $F$ its Fano variety of lines. Then $F$ admits a rational Lagrangian fibration if and only if $X$ is special of discriminant $d=2n^2$ for some integer $n$.
\end{prop}

The first example is $d=8$ when $X$ contains a plane $P$, and in that case, the isotropic class on $F$ under the Beauville-Bogomolov form is $\alpha(Q)$ where $Q\subset X$ is a quadric surface residual to $P$. The next section contains an explicit construction of the rational Lagrangian fibration $F\dashrightarrow\bP^2$ induced by the complete linear system of $\alpha(Q)$; moreover, this construction works over an arbitrary field of characteristic different from $2$.

It is worth noting that over $\bC$, one can give an alternate description of the rational Lagrangian fibration by using a result from \cite{macri} that $F$ is birational to a moduli space of twisted torsion sheaves on a K3 surface supported on curves in a $2$-dimensional linear system. From this perspective, the rational Lagrangian fibration sends a line to the curve on which the corresponding twisted sheaf is supported. For details on the compatibility of these two constructions, see \cite[Lemma 3.3.2]{brookethesis}

\subsection{The rational Lagrangian fibration}
Over a field $k$ of characteristic not $2$, let $X$ be a general cubic fourfold containing a plane $P$, by which we mean that $X$ is smooth and contains no other plane meeting $P$. Projection from $P$ induces a quadric surface fibration $q\colon \Bl_PX\to P^\perp$ whose singular fibers lie over a sextic curve $\Delta\subset P^\perp$ called the discriminant of $q$. Note that when $\chara(k)=2$, the curve $\Delta$ is instead cubic; for details, see \cite[Sect. 4.1]{hassetttransobs}. The assumption that $X$ contains no plane meeting $P$ guarantees that the singular fibers of $q$ are no worse than cones and, by Lemma 2 in \cite{voisin}, that $\Delta$ is smooth. 

Let $F$ be the Fano variety of lines on $X$, a smooth fourfold. The formula $\pi([L])=[q(L)]$ defines a rational map $F\dashrightarrow(P^\perp)^*$ whose domain contains the open subscheme of lines in $X\setminus P$. We will show that the rational map $\pi$ fits in the commutative square
\begin{center}
\begin{tikzcd}
    \Bl_{P^*}F \arrow[d] \arrow[r] & M \arrow[d,"\rho"] \\
    F \arrow[r,dashed,"\pi"] & \bP^2
\end{tikzcd}
\end{center}
where $M$ is the Mukai flop of $F$ along $P^*$ and $\rho$ is fibered in torsors of abelian surfaces, proving Theorem~\ref{theo:lagrangian}.

Before proving that $\pi$ is rational Lagrangian, it is worth explaining the relationship between (closures of) fibers of $\pi$ and the content of Section~\ref{sect:threefold}. If $L\subset P^\perp$, then $\pi^{-1}([L])\subset F$ contains all the lines in $(X\setminus P)\cap H$ where $H$ is the hyperplane spanned by $P$ and $L$. In the language of Theorem~\ref{theo:fanothree}, $\pi^{-1}([L])$ is the subscheme $\cU$ of the Fano scheme of lines on $Y=X\cap H$. 

We proceed as follows: first, we verify that $\pi$ extends to $F\setminus P^*$; second, we show that the blowup $\Bl_{P^*}F$ resolves the indeterminacy of $\pi$; finally, we show that $\pi$ factors through the contraction of the exceptional divisor 
\[
E=\{([L],x)\;|\;x\in L\subset P\}\subset \Bl_{P^*}F
\]
via the second projection. From this, we obtain a Lagrangian fibration from the Mukai flop of $F$ along $P^*$ whose fibers exhibit the arithmetic of the torsors described in Section~\ref{sect:threefold}.

\begin{lemm}\label{lemm:uniquehyp}
    For each $x\in P$, there is a unique hyperplane $H\supset P$ for which $X\cap H$ is singular at $x$.
\end{lemm}
\begin{proof}
    Since $X$ is a hypersurface, $X\cap H$ is singular at $x$ if and only if $H=T_xH$. When $x\in P$, one also has $T_xP\supset P$.
\end{proof}

\begin{lemm}\label{lemm:4to1}
    For each hyperplane $H\supset P$, the singular locus of $X\cap H$ along $P$ is zero-dimensional of length four.
\end{lemm}
\begin{proof}
    There is a morphism $g:P\to(P^\perp)^*$ sending $x\mapsto([T_xX\cap P^\perp])$, and the fiber of $g$ over $[L]\in(P^\perp)^*$ is the singular locus of $X\cap H$ along $P$, where $H$ is the hyperplane spanned by $P$ and $L$. Since any cubic threefold containing a plane has singularities along the plane, $g$ is surjective. Hence $g$ is finite, as is any surjective endomorphism of $\bP^2$. We end by applying Lemma~\ref{lemm:singthreefold} (i).
\end{proof}

Thus for each hyperplane $H\subset P$, the cubic threefold $Y=X\cap H$ containing $P$ is general in the sense of Section~\ref{sect:threefold}.

\begin{prop}\label{prop:domainpi}
The domain of $\pi$ is $F\setminus P^*$.
\end{prop}

The claim follows immediately from the next two lemmas.

\begin{lemm}\label{lemm:extend}
    Let $M$ be a line in $X$ meeting $P$ in exactly one point $x$. Then $[M]$ belongs to the domain of $\pi$, and $\pi([M])=[T_xX\cap P^\perp]$.
\end{lemm}
\begin{proof}
    Let 
    \[
    \cH=\{(L,H)\;|\;L,P\subset H\}\subset F\times(\bP^5)^*,
    \]
    let $p_1$ be the first projection, and let $p_2:\cH\to(P^\perp)^*$ send $(L,H)\mapsto[H\cap P^\perp]$. If $L\subset X\setminus P$, then $p_1^{-1}([L])=\{(L,H)\}$ where $H$ is the hyperplane spanned by $L$ and $P$, and $p_2(L,H)=\pi([L])$, which verifies that the diagram below commutes.
    \begin{center}
        \begin{tikzcd}
            & \cH \arrow[dl,"p_1" above] \arrow[dr,"p_2"] &\\ 
            F \arrow[rr,dashed,"\pi" below] & & (P^\perp)^*
        \end{tikzcd}
    \end{center}
    For $\ell\subset P^\perp$, let $H$ be the hyperplane spanned by $\ell$ and $P$, let $Y=X\cap H$, and let $F(Y)$ be the Fano scheme of lines on $Y$. We may identify $p_2^{-1}([\ell])$ with the $F(Y)$, and the restriction of $p_1$ to $p_2^{-1}([\ell])$ is an embedding. Let $Z=P\cap\mathrm{Sing}(Y)$, and let $P^*$, $\cF$, and $\oU$ be the three components of $F(Y)$, as described in Theorem~\ref{theo:fanothree}. We have $\cU=\pi^{-1}([\ell])$, so $[M]\in\overline{\pi^{-1}([\ell])}$ if and only if $[M]\in\oU\cap\cF$, which by Lemma~\ref{lemm:ubarf} holds if and only if $M\cap Z\neq\emptyset$. Since $M\cap P=\{x\}$, we have $M\cap Z\neq\emptyset$ if and only if $Y$ is singular at $x$, which occurs precisely when $H=T_xX$. Thus $[M]$ lies in the closure of the $\pi^{-1}([\ell])$ if and only if $\ell=T_xX\cap H$. Since $F$ is normal, $\pi$ extends to $[M]$.
\end{proof}

\begin{lemm}\label{lemm:doesntextend}
Each point $[L]\in P^*$ lies in the closures of a pencil of fibers of $\pi$. Specifically, $[L]\in\overline{\pi^{-1}([M])}$ if and only if $M=T_xX\cap P^\perp$ for some $x\in L$.
\end{lemm}
\begin{proof}
    Suppose $M\subset P^\perp$, let $H$ be the hyperplane spanned by $P$ and $M$, let $Y=X\cap H$, and let $F(Y)$ be the Fano scheme of lines on $Y$. As usual, we refer to the components of $F(Y)$ as $\oU$, $\cF$, and $P^*$. Since $L\subset P$, we have $[L]\in\oU$ if and only if $L$ passes through a node of $Y$ by Lemma~\ref{lemm:ubarp}, and this is the case if and only if $H=T_xX$ for some $x\in L$. This completes the argument since $\overline{\pi^{-1}([M])}=\oU$.
\end{proof}

Lemma~\ref{lemm:doesntextend} helps to describe the resolution of indeterminacy of $\pi$:

\begin{lemm}\label{lemm:blowupresolves}
The blowup of $F$ along $P^*$ resolves the indeterminacy of $\pi$.
\end{lemm}
\begin{proof}
    Let $\Gamma\subset F\times(P^\perp)^*$ be the graph of $\pi$ along with its projections $p_1$ and $p_2$, and note that since $\pi$ does not extend to any point on $P^*$, the projection $p_1$ factors through $\Bl_{P^*}F$. By Lemma~\ref{lemm:doesntextend}, the exceptional divisor $E'$ of $\Gamma\to F$ can be identified with 
    \[
    \{([L],[T_xX\cap P^\perp])\;|\;x\in L\},
    \]
    and the induced map 
    \[
    E'\to E=\{([L],x)\;|\;x\in L\subset P\}\subset\Bl_{P^*}F
    \]
    is an isomorphism. Hence $\Gamma\simeq\Bl_{P^*}F$, as needed.
\end{proof}

As mentioned prior, the exceptional divisor $E\subset\Bl_{P^*}F$ can be contracted via the projection $E\to P$ to give a blowdown $q:\Bl_{P^*}F\to M$ where $M$ is again a smooth variety, called the Mukai flop of $F$ along $P^*$.

\begin{prop}\label{prop:contract}
The morphism $p_2\colon\Bl_{P^*}F\to(P^\perp)^*$ factors through the blowdown $q\colon\Bl_{P^*}F\to M$.
\end{prop}
\begin{proof}
Let $[N]\in(P^\perp)^*$, and let $H$ be the hyperplane spanned by $P$ and $N$. By Lemmas~\ref{lemm:doesntextend} and~\ref{lemm:blowupresolves}, 
\[
p_2^{-1}([N])\cap E=\bigcup_{x\in P\cap\mathrm{Sing}(X\cap H)}\{([L],x)\;|\;x\in L\subset P\},
\]
so $p_2^{-1}([N])\cap E$ consists of a union of fibers of $q:\Bl_{P^*}F\to M$.
\end{proof}

We arrive at the commutative square presented in the statement of Theorem~\ref{theo:lagrangian}. It remains to show that the smooth fibers of $\rho\colon M\to(P^\perp)^*$ are torsors of abelian surfaces.

\begin{prop}
    Let $L\subset P^\perp$ be a line transverse to $\Delta$. Then there is a double cover $C$ of $L$ branched over $\Delta$ for which $\rho^{-1}([L])$ is a $\Pic_C^0$-torsor. Moreover, $2[\rho^{-1}([L])]=[\Pic^1_C]$ in the Weil-Ch\^atelet group.
\end{prop}
\begin{proof}
    Let $H$ be the hyperplane spanned by $P$ and $L$. Then $Y=X\cap H$ is a general cubic threefold containing $P$ by Lemma~\ref{lemm:4to1}. The quadric surface fibration $\Bl_PY\to L$ has discriminant locus $L\cap\Delta$, and let $C$ be the double cover of $L$ parametrizing lines in the fibers of $q$. Let $F(Y)$ be the Fano variety of lines on $Y$, and let $\oU$ be the closure of the set of lines in $Y\setminus P$. Note that $\oU=\overline{\pi^{-1}([L])}$, and from the natural embedding $\Bl_{\oU\cap P^*}\oU\subset\Bl_{P^*}F$, we identify
    \[
    (\pi\circ p)^{-1}([L])=\Bl_{\oU\cap P^*}\oU.
    \]
    By the proof of Proposition~\ref{prop:contract}, the restriction of $q$ to $(\pi\circ p)^{-1}([L])$ contracts a line in $\oU$ for each point in $P\cap \mathrm{Sing}(X\cap H)$. The same lines are contracted by the defining morphism $\Bl_{\oU\cap P^*}\oU\to T$ from Lemma~\ref{lemm:graphpsiz}. Hence we can identify $\rho^{-1}([L])$ with the $\Pic_C^0$-torsor $T$ and apply Theorem~\ref{theo:fanothree}.
\end{proof}


\bibliographystyle{plain}
\bibliography{bibliography}

\begin{thebibliography}{10}

\bibitem{barth}
Wolf Barth and Antonius Van~de Ven.
\newblock Fano varieties of lines on hypersurfaces.
\newblock {\em Arch. Math. (Basel)}, 31(1):96--104, 1978/79.

\bibitem{beauville}
Arnaud Beauville and Ron Donagi.
\newblock La vari\'{e}t\'{e} des droites d'une hypersurface cubique de
  dimension {$4$}.
\newblock {\em C. R. Acad. Sci. Paris S\'{e}r. I Math.}, 301(14):703--706,
  1985.

\bibitem{benoistwitt}
Olivier Benoist and Olivier Wittenberg.
\newblock The {C}lemens-{G}riffiths method over non-closed fields.
\newblock {\em Algebr. Geom.}, 7(6):696--721, 2020.

\bibitem{brookethesis}
Corey Brooke.
\newblock {\em Lines on cubic threefolds and fourfolds containing a plane}.
\newblock PhD thesis, University of Oregon, 2023.

\bibitem{clemens}
C.~Herbert Clemens and Phillip~A. Griffiths.
\newblock The intermediate {J}acobian of the cubic threefold.
\newblock {\em Ann. of Math. (2)}, 95:281--356, 1972.

\bibitem{eisenbud}
David Eisenbud.
\newblock {\em Commutative algebra}, volume 150 of {\em Graduate Texts in
  Mathematics}.
\newblock Springer-Verlag, New York, 1995.
\newblock With a view toward algebraic geometry.

\bibitem{3264}
David Eisenbud and Joe Harris.
\newblock {\em 3264 and all that---a second course in algebraic geometry}.
\newblock Cambridge University Press, Cambridge, 2016.

\bibitem{hadan}
Ingo Hadan.
\newblock Tangent conics at quartic surfaces and conics in quartic double
  solids.
\newblock {\em Math. Nachr.}, 210:127--162, 2000.

\bibitem{hartshorne}
Robin Hartshorne.
\newblock {\em Algebraic geometry}.
\newblock Graduate Texts in Mathematics, No. 52. Springer-Verlag, New
  York-Heidelberg, 1977.

\bibitem{hassett}
Brendan Hassett.
\newblock Special cubic fourfolds.
\newblock {\em Compositio Math.}, 120(1):1--23, 2000.

\bibitem{hassetttschinkel}
Brendan Hassett and Yuri Tschinkel.
\newblock Spaces of sections of quadric surface fibrations over curves.
\newblock {\em Compact moduli spaces and vector bundles}, 564:227--249, 2012.

\bibitem{hassetttschinkelrationality}
Brendan Hassett and Yuri Tschinkel.
\newblock Rationality of complete intersections of two quadrics over nonclosed
  fields.
\newblock {\em Enseign. Math.}, 67(1-2):1--44, 2021.
\newblock With an appendix by Jean-Louis Colliot-Th\'{e}l\`ene.

\bibitem{hassetttransobs}
Brendan Hassett, Anthony V\'{a}rilly-Alvarado, and Patrick Varilly.
\newblock Transcendental obstructions to weak approximation on general {K}3
  surfaces.
\newblock {\em Adv. Math.}, 228(3):1377--1404, 2011.

\bibitem{huybrechtscubic}
Daniel Huybrechts.
\newblock The geometry of cubic hypersurfaces.
\newblock 2022.

\bibitem{hwang}
Jun-Muk Hwang.
\newblock Base manifolds for fibrations of projective irreducible symplectic
  manifolds.
\newblock {\em Invent. Math.}, 174(3):625--644, 2008.

\bibitem{macri}
Emanuele Macr\`i and Paolo Stellari.
\newblock Fano varieties of cubic fourfolds containing a plane.
\newblock {\em Math. Ann.}, 354(3):1147--1176, 2012.

\bibitem{markmantorelli}
Eyal Markman.
\newblock A survey of {T}orelli and monodromy results for
  holomorphic-symplectic varieties.
\newblock In {\em Complex and differential geometry}, volume~8 of {\em Springer
  Proc. Math.}, pages 257--322. Springer, Heidelberg, 2011.

\bibitem{matsushitafiber}
Daisuke Matsushita.
\newblock On fibre space structures of a projective irreducible symplectic
  manifold.
\newblock {\em Topology}, 38(1):79--83, 1999.

\bibitem{matsushitadivisor}
Daisuke Matsushita.
\newblock On isotropic divisors on irreducible symplectic manifolds.
\newblock {\em Higher dimensional algebraic geometry---in honour of {P}rofessor
  {Y}ujiro {K}awamata's sixtieth birthday}, 74:291--312, 2017.

\bibitem{mongardi}
Giovanni Mongardi and Antonio Rapagnetta.
\newblock Monodromy and birational geometry of {O}'grady's sixfolds.
\newblock {\em Journal de math\'ematiques pures et appliqu\'ees},
  146(1):31--68, 2021.

\bibitem{voisin}
Claire Voisin.
\newblock Th\'{e}or\`eme de {T}orelli pour les cubiques de {${\bf P}^5$}.
\newblock {\em Invent. Math.}, 86(3):577--601, 1986.

\bibitem{wang}
Xiaoheng Wang.
\newblock Maximal linear spaces contained in the based loci of pencils of
  quadrics.
\newblock {\em Algebr. Geom.}, 5(3):359--397, 2018.

\end{thebibliography}

\end{document}